\documentclass[a4paper,fleqn]{cas-sc}

\usepackage[numbers]{natbib}
\usepackage{multirow}
\usepackage{amsmath}
\usepackage{lineno}

\usepackage{graphicx}
\usepackage{graphics}
\usepackage{epsf}
\usepackage{latexsym}
\usepackage{showidx}
\usepackage{latexsym}
\usepackage{amssymb}
\usepackage{pgfplots}
\usetikzlibrary{pgfplots.groupplots}
\usepackage{upquote}
\usepackage{graphicx,color}
\usepackage{algcompatible}
\usepackage{latexsym}
\usepackage{pdflscape}
\usepackage{bigfoot}

\usepackage{cleveref} 
\crefname{algocf}{algorithm}{algorithms}

\newcommand{\eqnref}[1]{(\ref{#1})}

\newtheorem{theorem}{Theorem}

\newdefinition{definition}{Definition}
\newdefinition{remark}{Remark}
\newdefinition{example}{Example}
\newproof{proof}{Proof}
\usepackage{todonotes}
\usepackage{algorithm}
\usepackage{algpseudocode}

\newcommand{\ignore}[1]{}

\def\tsc#1{\csdef{#1}{\textsc{\lowercase{#1}}\xspace}}
\tsc{WGM}
\tsc{QE}

\begin{document}
\let\WriteBookmarks\relax
\def\floatpagepagefraction{1}
\def\textpagefraction{.001}

\shorttitle{Sequential Multilinear Nystr\"om}    

\shortauthors{}  

\title [mode = title]{A sequential multilinear Nyström algorithm for streaming low-rank approximation of tensors in Tucker format}  

\tnotemark[1] 

\tnotetext[1]{} 

%

\author[1]{Alberto Bucci}[orcid=0009-0002-7550-6025]


\ead{alberto.bucci@phd.unipi.it}


\credit{}

\affiliation[1]{organization={Department of Mathematics, University of Pisa},
            city={Pisa},
            postcode={56127}, 
            country={Italy}}

\author[2]{Behnam Hashemi}[orcid=0000-0002-8847-2058]

\cormark[1]


\ead{b.hashemi@le.ac.uk}


\credit{}

\affiliation[2]{organization={School of Computing and Mathematical Sciences, University of Leicester}, 
            city={Leicester},
            postcode={LE1 7RH}, 
            country={United Kingdom}
            }

\cortext[1]{Corresponding author}

\fntext[1]{}


\begin{abstract}
We present a sequential version of the multilinear Nyström algorithm which is suitable for low-rank Tucker approximation of tensors given in a streaming format. Accessing the tensor $\mathcal{A}$ exclusively through random sketches of the original data, the algorithm effectively leverages structures in $\mathcal{A}$, such as low-rankness, and linear combinations. We present a deterministic analysis of the algorithm and demonstrate its superior speed and efficiency in numerical experiments including an application in video processing. \nocite{*}
\end{abstract}



\maketitle

\section{Introduction} \label{intro}

Given a tensor $\mathcal{A} \in \mathbb{R}^{n_1\times n_2 \times \dots \times n_d}$, the Tucker decomposition refers to a family of representations that factorize $\mathcal{A}$ into the multilinear product of a core tensor $\mathcal{C}\in \mathbb{R}^{r_1\times r_2 \times \dots \times r_d}$ and factor matrices $F_k\in \mathbb{R}^{n_k\times r_k}\ (r_k\leq n_k) $ along each mode $k = 1, 2, \dots, d$, i.e.,
\[
\mathcal{A} = \mathcal{C} \times_1 F_1\times_2 F_2 \dots \times_d F_d := \mathcal{C}\times_{k=1}^d F_k.
\]
See Section \ref{sec:preliminaries} for the definition of the mode-$k$ product.

We address the problem of finding an efficient randomized algorithm for the streaming low-rank approximation in this format exploiting only tensor mode products, or contractions. The algorithm can be interpreted as a sequential extension of the multilinear Nystr\"om (MLN) \cite{bucci2023multilinear} algorithm and of \cite[Alg. 4.3]{Another_Tucker}. By iterating the compression on sketched versions of the original tensor it avoids several operations that can be computational bottlenecks.

In this sense SMLN is similar to ST-HOSVD \cite{sthosvd}, randomized ST-HOSVD \cite{R-ST-HOSVD, Che21, Saibaba_Minster_Arvind}, and RTSMS \cite{hashemi2023rtsms}, but unlike these methods, is suitable for the streaming model \cite{clarkson2009numerical}. 
It can also be interpreted as a higher-order variant of the generalized Nystr\"om (GN) method \cite{tropp2017practical,nakatsukasa2020fast}.

The algorithm is highly efficient in the dense case and, as shown in the numerical experiments, offers a significant speed-up in the streaming setting, which is the main focus of this work. See \cite{Thanh} for
a readable overview of various streaming tensor decompositions. 

\section{Notations and preliminaries}\label{sec:preliminaries}

In this section, we introduce a few concepts and notations used
throughout the paper. A \textit{tensor} 
$\mathcal{A}\in\mathbb{R}^{n_1\times \dots \times n_d}$ is a $d$ dimensional array with entries
$a_{i_1i_2\dots i_d}$. 
The symbols used for transposition, Moore-Penrose inverse and 
Kronecker products of matrices are $T$, $\dagger$ and $\otimes$ respectively.
We use $\|\cdot\|_F$ for the Frobenius norm and $\|\cdot\|_2$ for
the spectral norm. We denote by $Q = \mathrm{orth}(X)$ the $Q$ factor of an economy
size QR factorization of a matrix $X$ with more rows than columns.
The \textit{mode}-$k$ \textit{matricization} of {$\mathcal{A}$} is the $n_k\times \prod_{i\neq k} n_i$ matrix $\mathcal{A}_k$. The \textit{mode}-$k$ \textit{product} of a tensor {$\mathcal{A}$} and a matrix $X \in \mathbb{R}^{m\times n_k}$  is denoted by $\mathcal{A}\times_k X$. Accordingly the $n_k\times m$ matrix $\mathcal{A}\times_{-k} Y$ denotes the mode product of $\mathcal{A}$ and $Y\in \mathbb{R}^{\prod_{i\neq k} n_i \times m}$ along all but the $k$th index. 

The mode-$k$ product along all dimensions can be effectively expressed by leveraging a mix of matricizations and Kronecker products as follows:
\[
(\mathcal{A}\times_1 X_1 \times \dots \times X_d)_{k} = X_k \mathcal{A}_{k} (X_d \otimes \dots\otimes X_{k+1} \otimes X_{k-1} \otimes \dots \otimes X_1)^T.
\]
Finally $X_{\otimes_{< k}} := X_{k-1} \otimes \dots \otimes X_1$. 

\section{Sequential Multilinear Nystr\"om (SMLN)}

Since SMLN is a higher-order variant of GN and a sequential version of MLN, let us first briefly review these two algorithms. Given a matrix $A \in \mathbb{R}^{m \times n}$, the GN algorithm computes an approximation of rank $r$, denoted as $\widehat{A}$, as follows. Initially, two random dimension reduction maps (DRMs) or sketchings are generated: $X \in \mathbb{R}^{n \times r}$ and $Y \in \mathbb{R}^{m \times (r+\ell)}$, where $\ell$ is an oversampling parameter to enhance accuracy and stability, then the approximant $\widehat{A}= AX (Y^TAX)^\dagger Y^TA$ if formed. To be precise the algorithm first computes the terms $AX$, $Y^TA$ and $Y^TAX$ and finally forms the rank $r$ factorization $\widehat{A} \approx (AXR^{-1}) (Q^T Y^T A)$, where $QR$ denotes the economy-sized QR of $Y^TAX$. 
The MLN algorithm extends this approach to tensors. Given a tensor {$\mathcal{A}$} the MLN algorithm first generates random sketchings $X_1,\dots, X_d$ and $Y_1, \dots, Y_d$ of size $\prod_{i\neq k}^d n_i\times r_k$ and $n_k \times (r_k+\ell_k)$ respectively, then computes the small tensor $\mathcal{A}\times_{k=1}^d Y_k^T$ and the matrices $\mathcal{A}_k X_k$ and $Y_k^T \mathcal{A}_k X_k$. Finally, it forms the multilinear rank $(r_1, \dots, r_d)$ approximation $\widehat{\mathcal{A}} = \left(\mathcal{A}\times_{k=1}^d Y_k^T\right)\times_{k=1}^d \mathcal{A}_k X_k( Y_k^T\mathcal{A}_k X_k)^\dagger$. For dense tensors, the most expensive part of MLN is the computation of the contractions $\mathcal{A}_k X_k$, for $k= 1, \dots, d$ which requires $\sum_{k=1}^d (2(n_1 \cdots n_{k-1} n_{k+1} \cdots n_d)-1) (r_k n_k) \sim 2dr n^d$ operations.
We reduce these costs by applying these contractions to progressively smaller tensors $\mathcal{A}\times_{i=1}^{k-1} Y_i^T$.
Recall that the MLN approximant can be written as $\widehat{\mathcal{A}}\times_{k=1}^d P_k$, where $P_k$ is the oblique projection $\mathcal{A}_k X_k( Y_k^T\mathcal{A}_k X_k)^\dagger Y_k^T$, we instead set

\begin{equation} \label{eq: projection}
    P_k=\mathcal{A}_k(I \otimes Y_{\otimes_{< k}}  ) X_k (Y_k^T \mathcal{A}_k(I \otimes Y_{\otimes_{< k}})X_k)^\dagger Y_k^T
\end{equation}
and define the SMLN approximant $\widehat{\mathcal{A}}$ of $\mathcal{A}$ as 

\begin{equation}\label{eq: smln_projection}
    \widehat{\mathcal{A}} := \mathcal{A}\times_{k=1}^d P_k = \left(\mathcal{A}\times_{k=1}^d Y_k^T\right) \times_{k=1}^d \mathcal{A}_k(I \otimes Y_{\otimes_{< k}}  )X_k (Y_k^T\mathcal{A}_k(I \otimes Y_{\otimes_{< k}})X_k)^\dagger.
\end{equation}

The overall cost is effectively reduced as the contractions involve the smaller matrices $\mathcal{A}_k(I \otimes Y_{\otimes_{< k}}  )$ rather than $\mathcal{A}_k$.
We summarize the method in Alg. \ref{alg:sequential MLN}. Note that, as an input, the processing order $p$ can be chosen.

\begin{algorithm}
\caption{Sequential Multilinear Nystr\"om} 
\vspace{2mm}
\textbf{Input} \hspace{6mm} $\mathcal{A}\in \mathbb{R}^{n_1\times\dots \times n_d}$, multilinear rank $r = (r_1, \dots, r_d) \leq (n_1,\dots, n_d)$, oversampling vector $\ell= (\ell_1,\dots, \ell_d)$,\\
\phantom{\textbf{Input}} \hspace{6mm} ordering $p = (p_1, \dots, p_d)$ with skip modes $p_{j+1}, p_{j+2}, \dots, p_d$.\\
\textbf{Output} \hspace{3mm} Low-rank Tucker approximant $\widehat{\mathcal{A}}$ of $\mathcal{A}$.\\
\phantom{\textbf{Input}} \hspace{6mm} Set $\mathcal{B} = \mathcal{A}$.\\
\phantom{\textbf{Input}} \hspace{6mm} \textbf{for} $k = 1, \dots, j$\\
\mbox{}~~~~~~~~~~~~~~~~~~~~Draw random matrices $X_{p_k}\in \mathbb{R}^{\prod_{i<k} (r_{p_i}+\ell_{p_i})\prod_{i>k} n_{p_i}\times r_{p_k}}$ and $Y_{p_k}\in \mathbb{R}^{n_{p_k}\times (r_{p_k}+\ell_{p_k})}$.\\
\mbox{}~~~~~~~~~~~~~~~~~~~~Compute $B_{p_k} X_{p_k}$, $Y_{p_k}^T B_{p_k}$ and $Y_{p_k}^T B_{p_k} X_{p_k}$.\\
\mbox{}~~~~~~~~~~~~~~~~~~~~Compute $QR$ factorization $Y_{p_k}^T B_{p_k} X_{p_k} = Q_{p_k} R_{p_k}$.\\
\mbox{}~~~~~~~~~~~~~~~~~~~~Set $F_{p_k} = B_{p_k} X_{p_k} R_{p_k}^{-1}$ and $B_{p_k}= Y_{p_k}^T B_{p_k}$.\\ 
\phantom{\textbf{Input}} \hspace{6mm} \textbf{end}\\ 
\phantom{\textbf{Input}} \hspace{6mm} Compute  $\widehat{\mathcal{A}} = (\mathcal{B}\times_{k=p_1}^{p_j} Q_k^T)\times_{k=p_1}^{p_j} F_k$.
\vspace{2mm}
\label{alg:sequential MLN}
\end{algorithm}

The standard Tucker decomposition involves low-rank approximation across all of the $d$ modes resulting in $d$ factor matrices. However, in some cases, a {\em partial} Tucker decomposition is desired (see Sec. \ref{video:sec}), where the tensor is not compressed in certain modes referred to as skip modes. This means that there is no factor matrix in the skip modes or equivalently, the corresponding factor matrix is the identity matrix. In Alg. \ref{alg:sequential MLN}, the first $j$ modes are treated as the standard (non-skipped) modes, while the remaining $d-j$ modes are designated as the skip modes. That configuration yields a partial Tucker decomposition in the first $j$ modes. To obtain the standard Tucker decomposition, one simply sets $j = d$.

\section{SMLN in the streaming model} \label{streamable_imp:sec}
The aim of this section is to compute the SMLN approximation of a linear combination of tensors $\mathcal{A} = \lambda_1 \mathcal{H}_1 + \dots + \lambda_m \mathcal{H}_m$ of size $n_1 \times \dots \times n_d$ accordingly to the streaming model \cite{clarkson2009numerical}. In this model each $\mathcal{H}_k$ is processed and discarded before accessing the next one.

We divide the algorithm into two phases: a sketching phase and a recovery phase. In the sketching phase, we draw $2d$ random sketchings and compute the contractions. In the recovery phase, we use the outputs of the sketching phase to obtain the low-rank factorization.
For simplicity, we fix the processing order to $p = (1, \dots, d)$ and consider $j = d$ which means we do not skip any of the modes.

In particular, given a set of target ranks $r_1, \dots, r_d$ and a set of oversampling parameters $\ell_1, \dots, \ell_d$, in the sketching phase we first generate the sketching matrices
    \[X_k\in \mathbb{R}^{\prod_{i<k} (r_{i}+\ell_{i})\prod_{i>k} n_{i}\times r_{k}}\quad\text{ and } \quad
    Y_{k}\in \mathbb{R}^{n_{k}\times (r_{k}+\ell_{k})},
    \]
and then, as described in Algorithm \ref{alg:SMLN_Sketch}, we use these sketchings to perform the contractions.

\begin{algorithm}
\caption{[$\mathcal{B}_d$, \{$\Omega_k$\}, \{$\Psi_k$\}] = SMLN\_Sketch($\mathcal{A}$, $\{X_k$\}, \{$Y_k$\}).} \label{alg:SMLN_Sketch}

\textbf{Input} Tensor $\mathcal{A}$, sketching matrices $X_k$ and $Y_k$.\\
\textbf{Output} Tensor $\mathcal{B}_d$, and matrices $\Omega_k$, $\Psi_k$.\\
\phantom{\textbf{Input}} \hspace{6mm} Set $\mathcal{B}_0 = \mathcal{A}$;\\
\phantom{\textbf{Input}} \hspace{6mm} \textbf{for} $k = 1, \dots, d$\\
    \mbox{}~~~~~~~~~~~~~~~~~~~~$\mathcal{B}_k = \mathcal{B}_{k-1}\times_k Y_k^T$, \quad ${\Omega}_k = \mathcal{B}_{k-1}\times_{-k} X_k$, \quad ${\Psi}_k = Y_k^T\Omega_k $,\\
\phantom{\textbf{Input}} \hspace{6mm} \textbf{end} 
 
\end{algorithm}

 More specifically, each $\Omega_k$ is of size $n_k \times r_k$, and each $\Psi_k$ is of size $(r_k + \ell_k) \times r_k$. In addition, $\mathcal{B}_d$ is a tensor of size $(r_1 + \ell_1) \times (r_2 + \ell_2) \times \dots \times (r_d + \ell_d)$. Finally, we compute the actual factors and the core tensor of the decomposition, see Alg \ref{alg:SMLN_Recovery}.

\begin{algorithm}
\caption{[$\mathcal{C}$, \{$F_k$\}] = SMLN\_Recovery($\mathcal{B}_d$, \{$\Omega_k$\}, \{$\Psi_k$\}) .}
\label{alg:SMLN_Recovery}

\textbf{Input} Sketched factors $\mathcal{B}_d$, $\Omega_k$, $\Psi_k$.\\
\textbf{Output}
 Tucker factors $\mathcal{C}$, $F_k$.\\
\phantom{\textbf{Input}} \hspace{6mm} \textbf{for} $k = 1, \dots, d$\\
    \mbox{}~~~~~~~~~~~~~~~~~~~~$[Q_k, R_k] = \mathrm{qr}(\Psi_k)$, \quad $F_k = \Omega_k R_k^\dagger$,\\
\phantom{\textbf{Input}} \hspace{6mm} \textbf{end} \\
    \phantom{\textbf{Input}} \hspace{6mm} $\mathcal{C} = \mathcal{B}_d\times_{k=1}^d Q_k^T$.
\end{algorithm}

To approximate the tensor $\mathcal{A}$ given as a stream of $\mathcal{H}_s$ we do as follows: for each $s=1,\dots, m$ we compute [$\mathcal{B}^{(s)}_d$, \{$\Omega_k^{(s)}$\}, \{$\Psi_k^{(s)}$\}]= SMLN\_Sketch($\mathcal{H}_s$, $\{X_k\}$, \{$Y_k$\}), next we set
$\mathcal{B}_d = \sum_{s = 1}^m \lambda_s \mathcal{B}_d^{(s)}$, $\Omega_k = \sum_{s = 1}^m \lambda_s \Omega^{(s)}_k$ and $\Psi_k = \sum_{s = 1}^m \lambda_s \Psi^{(s)}_k$ and finally we compute [$\mathcal{C}$, \{$F_k$\}] = SMLN\_Recovery($\mathcal{B}_d$, \{$\Omega_k$\}, \{$\Psi_k$\}).

Note that after each call of SMLN\_Sketch, the previous sketching can be added to the new one and discarded.

The cost of the most expensive operation in both Algorithms~\ref{alg:SMLN_Sketch} and~\ref{alg:SMLN_Recovery} is $(2(n_2 \cdots n_d)-1) (r_1 n_1)$. This corresponds to the operation of computing the $n_1 \times r_1$ matrix $\Omega_1$ in the very first step where {$\mathcal{A}$} is contracted with $X_1\in \mathbb{R}^{(n_2 n_3 \cdots n_d) \times r_1}$ in all-but-one of the modes. Due to the sequential nature of the algorithm, every other operation within both algorithms involves tensors and matrices of smaller size. This is especially important in the case of the recovery phase which merely 
involves small matrices. The most expensive operation in the recovery phase is the computation of the factor matrix $F_k$ whose cost is linear in terms of $n_k$ and cubic in terms of $r_k$.

\section{Analysis of SMLN}

In this section, we present a deterministic upper bound for the accuracy of the method. Given that the majority of the steps are derived from \cite{bucci2023multilinear} and the reference therein, we will provide only a brief overview.

\begin{theorem}
 \label{thm: deterministic_theorem}
    Let $\mathcal{A}\in \mathbb{R}^{n_1\times \dots \times n_d}$ be a tensor and let $\widehat{\mathcal{A}}$ be the output of Algorithm \ref{alg:sequential MLN}. Then, denoting with $Q_kR_k =  \mathcal{A}_k (I \otimes Y_{\otimes_{< k}}) X_k$ the economy-size QR and assuming $Y_k^T \mathcal{A}_k (I \otimes Y_{\otimes_{< k}}) X_k$ to be of full column-rank for each $k=1, \dots, d$, we have
\begin{equation}
\|\mathcal{A}-\widehat{\mathcal{A}}\|_F\leq \sum_{k=1}^d \|Q_{k \perp}^T \mathcal{A}_k(I \otimes Y_{\otimes_{< k}}) \|_F\|I - Q_k(Y_k^T Q_k)^\dagger Y_k^T\|_2\prod_{i=1}^{k-1}\|(Y_i^T Q_i)^\dagger\|_2.
\end{equation}
\end{theorem}
\begin{proof}
 
By \eqnref{eq: smln_projection} the approximant satisfies $\widehat{\mathcal{A}} = \mathcal{A}\times_{k=1}^d P_k$. By adding and subtracting terms of the form $\mathcal{A}\times_{i=1}^k P_i$, we get

    \begin{equation} \label{eq: sum_and_substract}
        \|\mathcal{A}-\hat{\mathcal{A}}\|_F\leq \sum_{k=1}^d\|\mathcal{A}\times_{i=1}^{k-1} P_i-\mathcal{A}\times_{i=1}^k P_i\|_F = \sum_{k=1}^d\|(\mathcal{A}_k(I \otimes P_{\otimes_{< k}}^T) -P_k\mathcal{A}_k(I \otimes P_{\otimes_{< k}}^T)\|_F .
    \end{equation}
The $s$th addend in the latter sum, by Equation \eqref{eq: projection}, is bounded by 
\begin{align} \label{eq: matricized_addend}
    &\|(\mathcal{A}_s(I \otimes Y_{\otimes_{< s}}) -P_s\mathcal{A}_s(I \otimes Y_{\otimes_{< s}})\|_F \prod_{k=1}^{s-1} \|\mathcal{A}_k(I \otimes Y_{\otimes_{< k}}) X_k (Y_k^T \mathcal{A}_k(I\otimes Y_{\otimes_{< k}})X_k)^\dagger\|_2.
\end{align}
Now, $\mathcal{A}_k(I\otimes Y_{\otimes_{< k}})X_k (Y_k^T \mathcal{A}_k (I \otimes Y_{\otimes_{< k}}) X_k)^\dagger=Q_kR_k(Y_k^TQ_kR_k)^\dagger$ and by the full-rank assumptions it simplifies to $Q_k(Y_k^TQ_k)^\dagger$ for $k=1, \dots d$. This allows us to rewrite \eqref{eq: matricized_addend} as
\begin{align*}
    &\|(\mathcal{A}_s(I \otimes Y_{\otimes_{< s}}) -Q_s(Y_s^T Q_s)^\dagger Y_s^T \mathcal{A}_s(I \otimes Y_{\otimes_{< s}})\|_F \prod_{k=1}^{s-1} \|Q_k (Y_k^T Q_k)^\dagger\|_2.\\
\end{align*}

\vspace{-0.7cm}
And since $ I -Q_s(Y_s^T Q_s)^\dagger Y_s^T  = (I -Q_s(Y_s^T Q_s)^\dagger Y_s^T) (Q_s Q_s^T + Q_{s \perp}Q_ {s \perp}^T)=(I - Q_s(Y_s^T Q_s)^\dagger Y_s^T) Q_{s \perp} Q_{s \perp}^T$, where $Q_{s \perp}$ is the orthogonal complement of $Q_s$, a straightforward computation leads to 

\begin{equation}
\|\mathcal{A}-\hat{\mathcal{A}}\|_F\leq \sum_{k=1}^d \|Q_{k \perp}^T \mathcal{A}_k(I \otimes Y_{\otimes_{< k}}) \|_F\|I - Q_k(Y_k^T Q_k)^\dagger Y_k^T\|_2\prod_{i=1}^{k-1}\|(Y_i^T Q_i)^\dagger\|_2.
\end{equation}
\end{proof}

Theorem \ref{thm: deterministic_theorem} provides a bound on the error of approximation which depends on multiple terms, all extensively analyzed in the literature \cite{randomproof}.
Indeed the term $\|Q_{k \perp }^T \mathcal{A}_k(I \otimes Y_{\otimes_{< k}}) \|_F$ corresponds to the error in Frobenius norm of the randomized SVD algorithm applied to the matrix $\mathcal{A}_k(I \otimes Y_{\otimes_{< k}})$ with sketching $X_k$, and strongly depends on the trailing singular values of $\mathcal{A}_k(I \otimes Y_{\otimes_{< k}})$ and also the terms of the form $\|I - Q_k(Y_k^T Q_k)^\dagger Y_k^T\|_2$ and $\|(Y_k^T Q_k)^\dagger\|_2$ arise in the analysis of the HMT algorithm \cite{randomproof}.

The fact that $\|Q_{k \perp}^T \mathcal{A}_k(I \otimes Y_{\otimes_{< k}}) \|_F$ depends on the singular values of $\mathcal{A}_k(I \otimes Y_{\otimes_{< k}})$, rather than those of $\mathcal{A}_k$, could be the only potential cause for concern. Indeed even if oblivious subspace embeddings preserve singular values with high probability \cite{meier2024fast}, the extent to which Kronecker embedding preserves the singular values is still an open question. For a more detailed discussion, see \cite{jin2021faster}.

\section{Numerical experiments}
We present two numerical experiments illustrating the performances of our method for streaming data.

The first experiment (Figure \ref{fig: artificial experiment 2}) compares the accuracy and computation time of MLN and SMLN for computing the Tucker approximation of an artificially generated order-4 tensor $\mathcal{A} = \sum_{s=1}^{15} \mathcal{H}_s$ of size $100 \times 100 \times 100 \times 100$. Each summand satisfies 
$\mathcal{H}_s = \mathcal{S}\times_{i=1}^4 Q_i^{(s)}$, where $\mathcal{S}$ is a super-diagonal tensor with diagonal elements $0.01^k$ and the $Q_i^{(s)}$ are orthogonal matrices generated independently according to the Haar distribution; leading to exponential decay in the singular values of each matricization of $H_s$. The rank of each decomposition is $(r,r,r,r)$ with $r$ varying between 10 and 55 with a step size of 5. While both MLN and SMLN achieve similar accuracy, the streaming SMLN method requires only 30\% of the computation time required by the non-sequential MLN method.

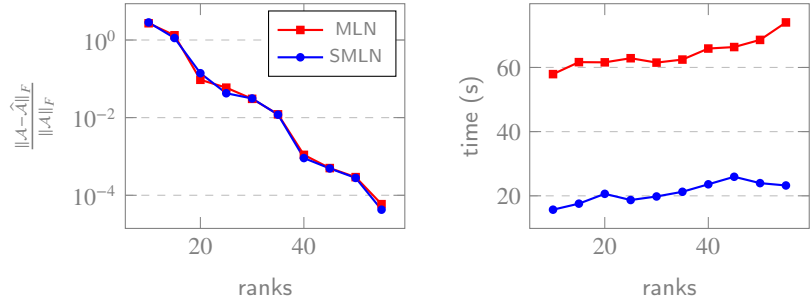
\begin{figure}
\centering
\begin{tikzpicture}
\begin{semilogyaxis}[
    x label style={at={(axis description cs:0.5,0)},anchor=north},
    xlabel={\small{ranks}},
    ylabel={{$\frac{\|\mathcal{A}-\widehat{\mathcal{A}}\|_F}{\|\mathcal{A}\|_F}$}},
    ymajorgrids=true,
    grid style=dashed,
    width=.32\linewidth,
    legend pos = north east,
    mark size = 1.2pt,
    restrict x to domain=0:125
]

\addplot [color = red, style = thick,mark = square*,  mark size=1.2pt] table [col sep=tab, x index = 0, y index = 1,x filter/.code={\pgfmathparse{mod(\coordindex,50)==0 ? x : inf}}] {accuracy_4D.dat};
\addplot [color = blue, style = thick,mark = *,  mark size=1.2pt] table [col sep=tab, x index = 0, y index = 2,x filter/.code={\pgfmathparse{mod(\coordindex,50)==0 ? x : inf}}] {accuracy_4D.dat};

\legend{\scriptsize{MLN}, \scriptsize{SMLN}}
\end{semilogyaxis}
\end{tikzpicture}~\begin{tikzpicture}
\begin{axis}[
    x label style={at={(axis description cs:0.5,0)},anchor=north},
    xlabel={\small{ranks}},
    ylabel={\small{time (s)}},
    y label style={at={(axis description cs:0.20,.5)},anchor=south},
    ymajorgrids=true,
    grid style=dashed,
    width=.32\linewidth,
    mark size = 1.2pt,
    restrict x to domain=0:125
]
\hspace*{0.5cm}
\addplot [color = red, style = thick,mark = square*,  mark size=1.2pt] table [col sep=tab, x index = 0, y index = 1,x filter/.code={\pgfmathparse{mod(\coordindex,50)==0 ? x : inf}}] {times_4D.dat};
\addplot [color = blue, style = thick,mark = *,  mark size=1.2pt] table [col sep=tab, x index = 0, y index = 2,x filter/.code={\pgfmathparse{mod(\coordindex,50)==0 ? x : inf}}] {times_4D.dat};
\end{axis}
\end{tikzpicture}
 \caption{Sequential Multilinear Nystr\"om with different choices of ranks $r$ and oversampling parameter $\ell = r/2$. The left plot shows the relative accuracy of the approximation in the Frobenius norm, and the right plot shows the execution time.} 
    \label{fig: artificial experiment 2}
\end{figure}

\subsection{Application to Tucker3 decomposition of order-4 tensors}  \label{video:sec}
Tensor decompositions have been used for video compression in the literature; see \cite{Zheng}, for instance. Here we explore a {\it partial} Tucker decomposition which takes advantage of our tensor-based sketching algorithm to efficiently process and analyze video data streams in a low-rank setting. Let $\mathcal{A}$ be an order-4 tensor of size $n_1 \times n_2 \times 3 \times n_4$ representing a video; the last dimension corresponds to the time variable. As time evolves, the frames of the video form a stream of tensors and one keeps collecting new frames into $\mathcal{A}$ hence representing it as 
\[
\mathcal{A} = \mathcal{A}_1 \circ e_1 + \mathcal{A}_2 \circ e_2  +  \dots + \mathcal{A}_{n_4} \circ e_{n_4}
\]
where $\mathcal{A}_s \circ e_s =: \mathcal{H}_s$ and $\circ$ represents outer product of an order-3 tensor $\mathcal{A}_s$ of size $n_1 \times n_2 \times 3$ and an $n_4 \times 1$ vector $e_s$ which is the $s$-th column of the $n_4 \times n_4$ identity matrix.

As the third dimension is already small, we aim at computing the Tucker3 decomposition $\mathcal{A} \approx \mathcal{C} \times_1 F_1 \times_2 F_2 \times_4 F_4,$ of the 4-th order tensor $\mathcal{A}$ where the core $\mathcal{C}$ is of size $r_1 \times r_2 \times 3 \times r_4$ and the three factor matrices are of size $n_1 \times r_1$, $n_2 \times r_2$ and $n_4 \times r_4$, respectively. 

We now examine our algorithm to compress the first 200 frames of the \texttt{big\_buck\_bunny\_scene} from the video dataset in \cite{Kiess}. We can consider the third mode as a skip mode and apply our algorithms with the ordering $(p_1, p_2, p_3, p_4) = (1, 2, 4, 3)$; see comments following Alg. \ref{alg:sequential MLN}. The video corresponds to a tensor of size $1080 \times 1920 \times 3 \times 200$. Using \cite{tensor_toolbox}, we run the streaming variants of both MLN and SMLN as outlined above with two multilinear ranks $\hat r = (200, 300, 3, 50)$ and $\check r = (400, 750, 3, 100)$. In the case of $\hat r$, MLN takes 110 seconds while SMLN needs 69 seconds. We aborted execution of MLN in the case of the larger rank $\check r$ as it took longer than 30 minutes, but SMLN runs in 8 minutes and 44 seconds. Note that the speed could be improved further if precisions lower than double are used, while still achieving a level of accuracy that is sufficient in the context of imaging. See Fig.~\ref{fig:video} for visual comparison of a few frames of the resulting compressions. The figure also reports the peak signal-to-noise ratio (PSNR) of each frame, with the same frame from the original video (i.e., the first picture in each row) as the reference. Note that a greater PSNR value indicates better image quality.

 \begin{figure} 
\includegraphics[width=0.92\textwidth]{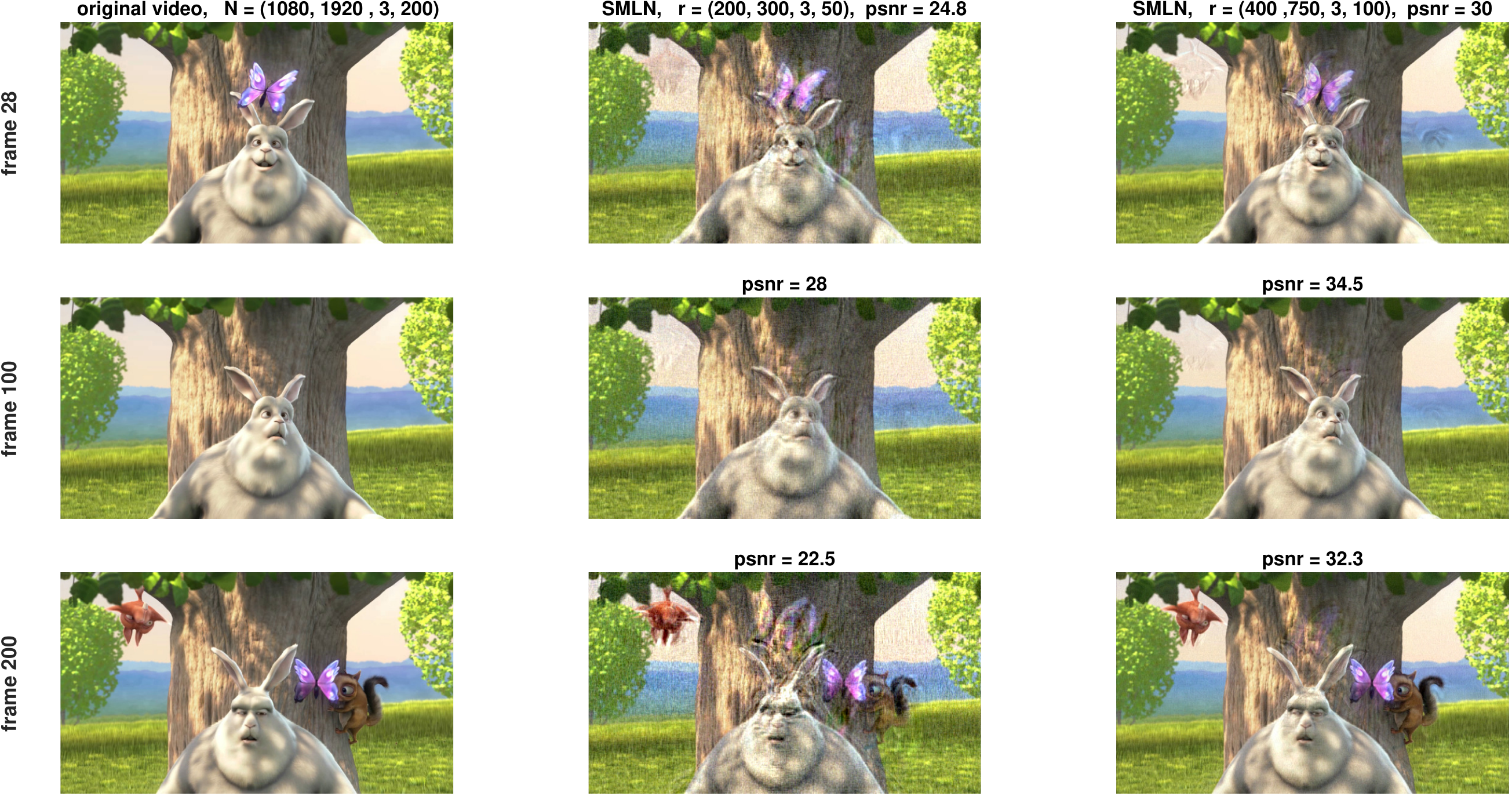}
 \caption{Streaming sequential multilinear Nystr\"om for Tucker3 decomposition of an order-4 tensor}
    \label{fig:video}
 \end{figure}
 \section*{Acknowledgment}
The authors appreciate insightful comments by the referees. They are also grateful for helpful discussions with Yuji Nakatsukasa, Leonardo Robol, and Gianfranco Verzella. Alberto Bucci is member of the INdAM Research group GNCS and has been supported by the PRIN 2022 Project “Low-rank Structures and Numerical Methods in Matrix and Tensor Computations and their Application” and by the MIUR Excellence Department Project awarded to the Department of Mathematics, University of Pisa, CUP I57G22000700001.

\bibliographystyle{model1-num-names}

\bibliography{biblio}
\end{document}